\documentclass{birkjour}
\usepackage{amsmath, amsthm, amsfonts, amssymb}
\usepackage{mathrsfs}
\usepackage{txfonts}
\usepackage{color}
\newtheorem{theorem}{Theorem}[section]
\newtheorem{definition}{Definition}[section]

\newtheorem{lemma}{Lemma}[section]
\newtheorem{prop}{Proposition}[section]
\theoremstyle{definition}

\allowdisplaybreaks

\usepackage[numbers]{natbib}
\setcitestyle{open={},close={}}

\def \l {\left}
\def \r {\right}

\def \bR {\Bbb R}

\def\be{\begin{equation}}
\def\ee{\end{equation}}
\def\ba{\begin{array}}
	\def\ea{\end{array}}
\def\bd{\begin{definition}}
	\def\ed{\end{definition}}
\def\bt{\begin{theorem}}
	\def\et{\end{theorem}}
\def\bc{\begin{corollary}}
	\def\ec{\end{corollary}}
\def\bl{\begin{lemma}}
	\def\el{\end{lemma}}
\def\bdm{\begin{displaymath}}
\def\edm{\end{displaymath}}

\begin{document}

%
%
%
%
%
%
%
%
%

\title[Multi-sublinear operators and their commutators]
 {Multi-sublinear operators and their commutators on product generalized mixed Morrey spaces}

\author[M. Wei]{Mingquan Wei}

\address{%
School of Mathematics and Stastics, Xinyang Normal University\\
Xinyang 464000, China}

\email{weimingquan11@mails.ucas.ac.cn}






\date{September 26, 2021}

\begin{abstract}
In this paper, we study the boundedness for a large class of multi-sublinear operators $T_m$
generated by multilinear Calder{\'o}n-Zygmund operators and their commutators $T^{b}_{m,i}~(i=1,\cdots,m)$  on the product generalized mixed Morrey spaces $M^{\varphi_1}_{\vec{q_1}}(\bR^n)\times\cdots\times M^{\varphi_m}_{\vec{q_m}}(\bR^n)$. We find the sufficient conditions on $(\varphi_1,\cdots,\varphi_m,\varphi)$ which ensure the boundedness of the operator $T_m$ from 
$M^{\varphi_1}_{\vec{q_1}}(\bR^n)\times\cdots\times M^{\varphi_m}_{\vec{q_m}}(\bR^n)$ to
$M^{\varphi}_{\vec{q}}(\bR^n)$.  Moreover, the sufficient conditions for the boundeness of $T^b_{m,i}$ from 
$M^{\varphi_1}_{\vec{q_1}}(\bR^n)\times\cdots\times M^{\varphi_m}_{\vec{q_m}}(\bR^n)$ to $M^{\varphi}_{\vec{q}}(\bR^n)$ are also studied.
As applications, we obtain the boundedness for the multi-sublinear maximal operator, the multilinear Calder{\'o}n-Zygmund operator and their commutators on product generalzied mixed Morrey spaces.
\end{abstract}

\maketitle
\section{Introduction}
\setcounter{equation}{0}
As we know, the multilinear Calder{\'o}n-Zygmund theory is a natural generalization of the linear case. The 
multilinear Calder{\'o}n-Zygmund operators were first introduced by Coifman
and Meyer [\cite{coifman1975commutators}] and were later systematically studied by Grafakos and Torres [\cite{grafakos2002multilinear}].

For $x\in\bR^n$, and $r>0$, let $B(x,r)$ be the open ball centered at $x$ with radius $r$, $B^c(x,r)$ denote its complement and $|B(x,r)|$ be the Lebesgue measure of the ball $B(x,r)$. We denote by $\vec{f}$ the $m$-tuple 
$(f_1,\cdots,f_m)$. Similarly, we also denote $\vec{y}=(y_1,\cdots,y_m)$ and $d\vec{y}=dy_1\cdots dy_m$.
For $n,m\in \mathbb{N}$, let $\bR^n$ be an $n$-dimensional Euclidean space and $(\bR^n)^m=\bR^n\times\cdots\times\bR^n$ be a $m$-fold
product spaces. 

Let $\vec{f}\in L^{1}_{\rm loc}(\bR^n)\times\cdots\times L^{1}_{\rm loc}(\bR^n)$. The multi-sublinear maximal operator $M_m$ is defined by 
$$M_m (\vec{f})(x)=\sup_{ r>0}\prod_{i=1}^m\frac{1}{|B(x,r)|}\int_B|f_i(y_i)|dy_i.$$
When $m=1$, we denote by $M_{1}=M$, the Hardy-Littlewood maximal operator.

For $p_i>1,~i=1,\cdots,m$, and $1/p=1/p_1+\cdots+1/p_m$, the boundedness of $M_m$ form product Lebesgue spaces $L^{p_1}(\bR^n)\times\cdots\times L^{p_m}(\bR^n)$ to $L^{p}(\bR^n)$ is obvious by using the boundedness of $M$ on $L^{p_i}(\bR^n)$ and H{\"o}lder's inequaly. One can also refer to [\cite{lerner2009new}] for the weighted boundedness of $M_m$.

In [\cite{grafakos2002multilinear}], Grafakos and Torres studied the multilinear Calder{\'o}n-Zygmund operators. We say $K_m$ is a multilinear Calder{\'o}n-Zygmund operator if for some $q_1,\cdots,q_m$ $\in [1,\infty)$ and $q\in (0,\infty)$ with $1/q=1/q_1+\cdots+1/q_m$, it can be extended to a bounded multilinear operator from $L^{q_1}(\bR^n)\times\cdots\times L^{q_m}(\bR^n)$ to $L^{q}(\bR^n)$, and  if there exists a kernel function $K(x,y_1,\cdots,y_m)$ defined off the diagonal $x=y_1=\cdots=y_m$ satisfying
$$|K(x,y_1,\cdots,y_m)|\leq C\l(\sum_{k=1}^m|x-y_k|\r)^{-mn},$$
such that $K_m$ can be written as 
\begin{equation*}
K_m (\vec{f})(x)=\int_{(\bR^n)^m}K(x,y_1,\cdots,y_m)f_1(y_1)\cdots f_m(y_m)dy_1\cdots y_m,
\end{equation*}
for $x\notin\cap_{j=1}^m\mathrm{supp}f_j$.
Moreover, for some $\epsilon>0$, $K(x,y_1,\cdots,y_m)$ satisfies the regularity condition
$$|K(x,y_1,\cdots,y_m)-K(x',y_1,\cdots,y_m)|\leq C\frac{|x-x'|^\epsilon}{\l(\sum_{k=1}^m|x-y_k|\r)^{mn+\epsilon}},$$
whenever $2|x-x'|\leq\max_{1\leq k\leq m}|x-y_k|$, and also for each fixed $k$ with $1\leq k\leq m$,
$$|K(x,y_1,\cdots,y_k,\cdots,y_m)-K(x,y_1,\cdots,y'_k,\cdots,y_m)|\leq C\frac{|y_k-y'_k|^\epsilon}{\l(\sum_{k=1}^m|x-y_k|\r)^{mn+\epsilon}},$$
whenever $2|y_k-y'_k|\leq\max_{1\leq k\leq m}|x-y_k|$.
Here and in what follows, $C$ is a constant which may vary from line to line.

Grafakos and Torres [\cite{grafakos2002multilinear}] proved that 
$K_m$ is bounded from $L^{p_1}(\bR^n)\times\cdots\times L^{p_m}(\bR^n)$ to $L^{p}(\bR^n)$ for $p_i>1~(i=1,\cdots,m)$
and $1/p=1/p_1+\cdots+1/p_m$, and bounded from  $L^{1}(\bR^n)\times\cdots\times L^{1}(\bR^n)$ to $L^{1/m,\infty}(\bR^n)$. We also refer the reader to [\cite{lerner2009new}] for the weighted estimates of $K_m$.

As extensions of the commutators of linear or sublinear operators, the multilinear commutators also attract much attention. 

Let $\vec{b}=(b_1,\cdots,b_m)\in L^1_{\rm loc}(\bR^n)\times\cdots\times L^1_{\rm loc}(\bR^n)$ and $\vec{f}\in L^{1}_{\rm loc}(\bR^n)\times\cdots\times L^{1}_{\rm loc}(\bR^n)$. The multi-sublinear maximal commutator $M^{\vec{b}}_{m}$ is defined by 
\begin{equation*}
M^{\vec{b}}_{m}(\vec{f})=\sum_{i=1}^{m}M^{\vec{b}}_{m,i}(\vec{f}),
\end{equation*}
where
$$M^{\vec{b}}_{m,i} (\vec{f})(x)=\sup_{ r>0}\frac{1}{|B(x,r)|}\int_B|b(x)-b(y_i)||f_i(y_i)|dy_i\times\sup_{ r>0}\prod_{k=1,k\neq i}^m\frac{1}{|B(x,r)|}\int_B|f_k(y_k)|dy_k.$$
When $m=i=1$, we denote by $M^b_{1,1}=M^b$, the maximal commutator of the Hardy-Littlewood maximal operator.

To study the boundedness of the commutators of some integral operators, we need the bounded mean oscillation space first introduced by John and Nirenberg [\cite{john1961functions}]. A locally integrable function $f$ belongs to BMO if 
\begin{equation*}
\|f\|_{\rm BMO}=
\sup_{x\in\bR^n,r>0}\frac{1}{|B(x,r)|}\int_{B(x,r)}|f(x)-f_{B(x,r)}|dx<\infty.
\end{equation*}
We will use the following notation: if $\vec{b}=(b_1,\cdots,b_m)\in{\rm BMO}^m$, then we denote by the norm $\|\vec{b}\|_{{\rm BMO}^m}=\sup_{i=1,\cdots,m}\|{b_i}\|_{{\rm BMO}}$.

Similar to the boundedness of $M_m$, for $\vec{b}\in {\rm BMO}^{m}$, one can get the boundedness of $M^{\vec{b}}_{m,i}$ from $L^{p_1}(\bR^n)\times\cdots\times L^{p_m}(\bR^n)$ to $L^{p}(\bR^n)$ for all $p_i>1,~i=1,\cdots,m$, and $1/p=1/p_1+\cdots+1/p_m$ by using the boundedness of $M^{b_i}$ on $L^{p_i}(\bR^n)$ (see [\cite{grafakos2008classical}]) and H{\"o}lder's inequaly. Therefore, $M^{\vec{b}}_{m}$ is also bounded from $L^{p_1}(\bR^n)\times\cdots\times L^{p_m}(\bR^n)$ to $L^{p}(\bR^n)$ under the same conditions.

Now we review the definition of the multilinear commutator of $K_m$, see, for instance, [\cite{lerner2009new}]. 

Let $\vec{b}\in L^1_{\rm loc}(\bR^n)\times\cdots\times L^1_{\rm loc}(\bR^n)$ and $\vec{f}\in L^{1}_{\rm loc}(\bR^n)\times\cdots\times L^{1}_{\rm loc}(\bR^n)$. The $m$-linear commutator $K^{\vec{b}}_{m}$ is defined by
\begin{equation*}
K^{\vec{b}}_{m}(\vec{f})=\sum_{i=1}^{m}K^{\vec{b}}_{m,i}(\vec{f}),
\end{equation*}
where
$$K^{\vec{b}}_{m,i} (\vec{f})=b_iK_m(f_1,\cdots,f_i,\cdots,f_m)-K_m(f_1,\cdots,b_if_i,\cdots,f_m).$$

The operator $K^{\vec{b}}_{m}$ was proved to be bounded from $L^{p_1}(\bR^n)\times\cdots\times L^{p_m}(\bR^n)$ to $L^{p}(\bR^n)$ for $p_i>1~(i=1,\cdots,m)$ and $1/p=1/p_1+\cdots+1/p_m$ by Lerner et al. [\cite{lerner2009new}].

The multi-sublinear maximal operator and multilinear Calder{\'o}n-Zygmund 
operators play a key role in the multilinear harmonic analysis, see [\cite{grafakos2002multilinear,lin2006multilinear}]. To unify the multi-sublinear maximal operator and the multilinear Calder{\'o}n-Zygmund operators, Lin and Lu [\cite{lin2006multilinear}] introduced a class of multi-sublinear operators $T_m$, which satisfies
that for any $\vec{f}\in L^{1}(\bR^n)\times\cdots\times L^{1}(\bR^n)$ with compact support and $x\notin\cap_{j=1}^m\mathrm{supp}f_j$,
\begin{equation}\label{D-multisub}
|T_m(\vec{f})(x)|\leq C\int_{(\bR^n)^m}\frac{|f_1(y_1)\cdots f_m(y_m)|}{|(x-y_1,\cdots,x-y_m)|^{mn}}dy_1\cdots dy_m,
\end{equation}
where $C$ is independent of $\vec{f}$ and $x$. 

Similarly, for $i\in\{1,\cdots,m\}$, $b\in L^1_{\rm loc}(\bR^n)$, we define a class of multi-sublinear commutators $T^b_{m,i}$, which satisfies
that for any $\vec{f}\in L^{1}(\bR^n)\times\cdots\times L^{1}(\bR^n)$ with compact support and $x\notin\cap_{j=1}^m\mathrm{supp}f_j$,
\begin{equation}\label{D-multisub-commu}
|T^b_{m,i}(\vec{f})(x)|\leq C\int_{(\bR^n)^m}\frac{|b(x)-b(y_i)|}{|(x-y_1,\cdots,x-y_m)|^{mn}}|f_1(y_1)\cdots f_m(y_m)|dy_1\cdots dy_m,
\end{equation}
where $C$ is independent of $\vec{f}$ and $x$.

Obviously, the multi-sublinear maximal operator and multilinear Calder{\'o}n-Zygmund 
operators satisfy the size condition (\ref{D-multisub}). Similarly, the commutators $M^{\vec{b}}_{m,i}$ and $K^{\vec{b}}_{m,i}$
satisfy the size condition (\ref{D-multisub-commu}).

The operator $T_m$ satisfying (\ref{D-multisub}) has been proved to be bounded from product Morrey spaces to  Morrey space [\cite{lin2006multilinear}], and bounded from product generalized Morrey spaces to  generalized Morrey space [\cite{2012SOME}] under some mild assumptions. In this paper, we will further prove the boundedness of the multi-sublinear operator $T_m$ satisfying 
condition (\ref{D-multisub}) generated by multilinear Calder{\'o}n-Zygmund operators from product generalized mixed Morrey spaces $M^{\varphi_1}_{\vec{q_1}}(\bR^n)\times\cdots\times M^{\varphi_m}_{\vec{q_m}}(\bR^n)$ to the generalized mixed Morrey space
$M^{\varphi}_{\vec{q}}(\bR^n)$. Moreover, we also establish the boundedness of the multi-sublinear commutator $T^b_{m,j}$ satisfying
condition (\ref{D-multisub-commu}) generated by the commutator of multilinear Calder{\'o}n-Zygmund operators from  $M^{\varphi_1}_{\vec{q_1}}(\bR^n)\times\cdots\times M^{\varphi_m}_{\vec{q_m}}(\bR^n)$ to
$M^{\varphi}_{\vec{q}}(\bR^n)$. Finally, as applications we apply our main theorems to the multi-sublinear maximal operator, the multilinear Calder{\'o}n-Zygmund operators and their commutators.

\section{Definitions and preliminaries}

Throughout the paper, we use the following notations.

The letter $\vec{p}$ denotes $n$-tuples of the numbers in $[0,\infty]$,~($n\geq1$),~$\vec{p}=(p_1,\cdots,p_n)$. By definition, the inequality $0<\vec{p}<\infty$ means $0<p_i<\infty$ for all $i$. For $1\leq\vec{p}\leq\infty$, we denote $\vec{p}'=(p'_1,\cdots,p'_n)$, where $p'_i$ satisfies
$\frac{1}{p_i}+\frac{1}{p'_i}=1$. By $A\lesssim B$, we
mean that $A\leq CB$ for some constant $C>0$, and $A\sim B$ means that $A\lesssim B$ and $B\lesssim A$.

Let $\mathcal{M}(\bR^n)$ be the class of all Lebesgue measurable functions.
By $L^{\vec{p}}_{{\rm loc}}(\bR^n)$, we mean the collection of all functions $f$ such that $f\chi_E\in L^{\vec{p}}(\bR^n)$ for all bounded and measurable sets $E\subseteq \bR^n$.  The letter $\mathbb{C}$ is the set of the complex numbers and  $\mathbb{N}$ is the set of all non-negative integers.

In order to study the local properties of solutions to the partial differential equations, Morrey [\cite{chiarenza1987morrey}] introduced the classical Morrey space $M^p_q(\bR^n)$, which consists of all 
functions $f\in L^q_{\rm loc}(\bR^n)$ with finite norm
$$\|f\|_{M^p_q}=\sup_{x\in\bR^n,r>0}|B(x,r)|^{\frac{1}{p}-{\frac{1}{q}}}\|f\|_{L^q(B(x,r))},$$
where $1\leq q\leq p\leq\infty$.
Note that $M^p_q(\bR^n)=L^p(\bR^n)$ when $p=q$, and $M^p_q(\bR^n)=L^{\infty}(\bR^n)$ when $p=\infty$. If $q>p$, then $M^p_q(\bR^n)=\Theta$, where $\Theta$ is the set of all functions equivalent to 0 on $\bR^n$. One can see [\cite{adams1975note,chiarenza1987morrey,peetre1969theory}] 
and some related papers for the boundedness of some classical operators on $M^p_q(\bR^n)$.

Recently, Guliyev et al. [\cite{burenkov2009necessary,guliyev2021regularity,guliyev2011boundedness,guliyev2013global}] extended $M^p_q(\bR^n)$ to the generalized Morrey spaces $M^\varphi_{q}(\bR^n)$ and studied the regularity of solutions of elliptic equations in divergence form in the generalized Morrey spaces.  

Let ~$1\leq q<\infty$ and $\varphi(x,r): \bR^n\times (0,\infty)\rightarrow(0,\infty)$ be a Lebesgue measurable function. A function $f\in\mathcal{M}(\bR^n)$ belongs to $M^\varphi_{q}(\bR^n)$ if it satisfies
\begin{eqnarray}
\|f\|_{M^\varphi_{q}}=\sup_{x\in\bR^n, r>0}\varphi(x,r)^{-1}|B(x,r)|^{-\frac{1}{q}}\|f\|_{L^q(B(x,r))}<\infty.
\end{eqnarray}
From the definition, one can recover the classical Morrey space $M^p_q(\bR^n)$ by taking $\varphi(x,r)=|B(x,r)|^{-1/p}$.
The boundedness of some sublinear operators and their commutators on $M^\varphi_{q}(\bR^n)$ was obtained by Guliyev et al. [\cite{guliyev2011boundedness}]. Moreover, the boundedness of some multi-sublinear operators from $M^{\varphi_1}_{q_1}(\bR^n)\times\cdots\times M^{\varphi_m}_{q_m}(\bR^n)$ to 
$M^{\varphi}_{q}(\bR^n)$ was also studied in [\cite{guliyev2015multilinear,ismayilovamulti2020,YU2014Boundedness}].

In 2019, Nogayama [\cite{nogayama2019mixed}] considered a new Morrey space, with the $L^p(\bR^n)$ norm replaced by the mixed Lebesgue norm $L^{\vec{q}}(\bR^n)$, which is call the mixed Morrey space. We point out that there exists another mixed Morrey space, which used the iteration of Morrey norm, introduced by Ragusa and Scapellato [\cite{ragusa2017mixed}]. For the boundedness of various operators on these mixed Morrey spaces of iteration type, see [\cite{anceschi2019operators,scapellato2020modified,scapellato2020riesz}].

To make the definition clear, we first recall the definition of mixed Lebesgue spaces.
Let ~$\vec{p}=(p_1,\cdots,p_n)\in(0,\infty]^n$. Then the mixed Lebesgue norm $\|\cdot\|_{\vec{p}}$ is defined by
\begin{eqnarray*}
	\|f\|_{\vec{p}}
	= \l(\int_{\bR}\cdots \l(\int_{\bR}\l(\int_{\bR}|f(x_1,x_2,\cdots,x_n)|^{p_1}dx_1\r)^{\frac{p_2}{p_1}}dx_2\r)^{\frac{p_3}{p_2}}\cdots dx_n\r)^{\frac{1}{p_n}}
\end{eqnarray*}
where $f: \bR^n \rightarrow \mathbb{C}$ is a measurable function. If $p_j=\infty$ for some $j=1,\cdots,n$, then we have to make appropriate modifications. We define the mixed Lebesgue space $L^{\vec{p}}(\bR^n)$
to be the set of all $f\in \mathcal{M}(\bR^n)$  with $\|f\|_{\vec{p}}<\infty$.
We refer the reader to [\cite{1961The}] for more details of mixed Lebesgue spaces.

Let $1\leq \vec{q}<\infty, 1\leq p<\infty$ and $n/p\leq \sum_{i=1}^n1/{q_i}$. A function $f\in\mathcal{M}(\bR^n)$ belongs to the mixed Morrey space $M^p_{\vec{q}}(\bR^n)$ if 
$$\|f\|_{M^p_{\vec{q}}}=\sup_{x\in\bR^n,r>0}|B(x,r)|^{\frac{1}{p}-{\frac{1}{n}\l(\sum_{i=1}^n\frac{1}{q_i}\r)}}\|f\chi_{B(x,r)}\|_{\vec{q}}<\infty.$$
Obviously, we return to the classical Morrey space $M^p_{q}(\bR^n)$ when $\vec{q}=q$. We point out that in [\cite{Nogayama2019Boundedness,nogayama2019mixed}], the author used the cubes to define the mixed Morrey spaces. One can verify the equivalence between the two definitions without any difficulty.

The mixed Morrey space is also an appropriate substitution of the Lebesgue space to study the mapping properties of some operators in harmonic analysis. In [\cite{Nogayama2019Boundedness,nogayama2019mixed}], the Hardy-Littlewood maximal operator $M$, the singular integral operators $K$, the fractional interal operator $I_\alpha$ and its commutator $I^b_\alpha$ were proved to be bounded in $M^p_{\vec{q}}(\bR^n)$. For the atom decomposition and the Olsen inequality for the mixed Morrey spaces, we refer the readers to [\cite{nogayama2020atomic}]. 

In order to unify the two types of Morrey spaces, we introduced generalized mixed Morrey spaces in [\cite{wei2021boundedness}], see also [\cite{zhang2021boundedness}]. Now we give the precise definition of  generalized mixed Morrey spaces.
\begin{definition}\label{GMM}
	Let $1\leq \vec{q}<\infty$, and $\varphi(x,r): \bR^n\times (0,\infty)\rightarrow(0,\infty)$ be a Lebesgue measurable function. A function $f\in\mathcal{M}(\bR^n)$ belongs to the mixed Morrey space $M^\varphi_{\vec{q}}(\bR^n)$ if 
	$$\|f\|_{M^\varphi_{\vec{q}}}=\sup_{x\in\bR^n,r>0}\varphi(x,r)^{-1}\|\chi_{B(x,r)}\|_{L^{\vec{q}}}^{-1}\|f\chi_{B(x,r)}\|_{\vec{q}}<\infty.$$	
\end{definition}	
Clearly, generalized mixed Morrey spaces contain generalized Morrey spaces and mixed Morrey spaces as special cases. In fact, if $\vec{q}=q$, then $M^\varphi_{\vec{q}}(\bR^n)=M^\varphi_{q}(\bR^n)$, and $M^\varphi_{\vec{q}}(\bR^n)=M^p_{\vec{q}}(\bR^n)$ when $\varphi(x,r)=|B(x,r)|^{-1/p}$.

A basic result is that the generalized mixed Morrey space $M^\varphi_{\vec{q}}(\bR^n)$ is a Banach space.
\begin{prop}
	Let $1\leq \vec{q}<\infty$, and $\varphi(x,r): \bR^n\times (0,\infty)\rightarrow(0,\infty)$ be a Lebesgue measurable function. Then $M^\varphi_{\vec{q}}(\bR^n)$ is a Banach space.	
\end{prop}
\begin{proof}
	Let $f_k$ be a Cauchy sequence in $M^\varphi_{\vec{q}}(\bR^n)$. Then for each $\epsilon>0$, there exists some integer $N(\epsilon)>0$, such that for all $i,j>N(\epsilon)$, we have $\|f_i-f_j\|_{M^\varphi_{\vec{q}}}<\epsilon$. Let $n_k=N(\frac{1}{2^k})$. For the above particular choice of $\epsilon$, there is a subsequence $\{f_{n_k}\}$, such that $\|f_{n_{k+1}}-f_{n_k}\|_{M^\varphi_{\vec{q}}}<\frac{1}{2^k}$.
	
	We set 
	$$f(x)=f_{n_1}(x)+\sum_{k=1}^\infty(f_{n_{k+1}}(x)-f_{n_k}(x)),\quad x\in\bR^n,$$
	and 
	$$h(x)=|f_{n_1}(x)|+\sum_{k=1}^\infty|f_{n_{k+1}}(x)-f_{n_k}(x)|,\quad x\in\bR^n.$$
	Define 
	$S_Nf=f_{n_1}+\sum_{k=1}^{N-1}(f_{n_{k+1}}-f_{n_k})=f_{n_N}$
	and 
	$S_Nh=|f_{n_1}|+\sum_{k=1}^{N-1}|f_{n_{k+1}}-f_{n_k}|.$ By using Minkowski's inequality, we obtain
	\begin{eqnarray*}
		\|S_Nh\|_{M^\varphi_{\vec{q}}}&=&\sup_{x\in\bR^n,r>0}\varphi(x,r)^{-1}\|\chi_{B(x,r)}\|_{L^{\vec{q}}}^{-1}\|h\chi_{B(x,r)}\|_{\vec{q}}\\
		&\leqslant&\sup_{x\in\bR^n,r>0}\varphi(x,r)^{-1}\|\chi_{B(x,r)}\|_{L^{\vec{q}}}^{-1}\|f_{n_1}\chi_{B(x,r)}\|_{\vec{q}}\\
		&+&\sum_{k=1}^{N-1}\sup_{x\in\bR^n,r>0}\varphi(x,r)^{-1}\|\chi_{B(x,r)}\|_{L^{\vec{q}}}^{-1}\|(f_{n_{k+1}}-f_{n_k})\chi_{B(x,r)}\|_{\vec{q}}\\
		&\leqslant&\|f_{n_1}\|_{M^\varphi_{\vec{q}}}+\sum_{k=1}^{N-1}\frac{1}{2^k}\leqslant \|f_{n_1}\|_{M^\varphi_{\vec{q}}}+1.
	\end{eqnarray*}
	As a consequence, $\|h\|_{M^\varphi_{\vec{q}}}<\infty$. Since $|f|\leqslant h$, we have $\|f\|_{M^\varphi_{\vec{q}}}\leqslant\|h\|_{M^\varphi_{\vec{q}}}<\infty$. Therefore, $f\in M^\varphi_{\vec{q}}(\bR^n)$.
	
	In view of
	\begin{eqnarray*}
		\lim_{N\rightarrow\infty}\|f-f_{n_N}\|_{M^\varphi_{\vec{q}}}
		&\leqslant&\lim_{N\rightarrow\infty}\sum_{k=N}^\infty\|f_{n_{k+1}}-f_{n_k}\|_{M^\varphi_{\vec{q}}}\\
		&\leqslant&\lim_{N\rightarrow\infty}\sum_{k=N}^\infty\frac{1}{2^k}
		=\lim_{N\rightarrow\infty}\frac{1}{2^{N-1}}=0,
	\end{eqnarray*}
	we deduce that the sequence $\{f_{n_k}\}$ converges to $f$ in $M^\varphi_{\vec{q}}(\bR^n)$. Thus, the Cauchy sequence $\{f_{k}\}$ also converges to $f$ in $M^\varphi_{\vec{q}}(\bR^n)$. We are done.
\end{proof}

In [\cite{wei2021boundedness}], the author obtained the boundedness of some sublinear operators on $M^\varphi_{\vec{q}}(\bR^n)$. In this paper, we will further extend the results in [\cite{wei2021boundedness}] to the multi-sublinear situation. 

\section{Boundedness of $T_m$ from 
	$M^{\varphi_1}_{\vec{q_1}}(\bR^n)\times\cdots\times M^{\varphi_m}_{\vec{q_m}}(\bR^n)$ to
	$M^{\varphi}_{\vec{q}}(\bR^n)$}
In this section, we investigate the boundedness of $T_m$ satisfying the size condition (\ref{D-multisub}) from  product generalized mixed Morrey spaces
$M^{\varphi_1}_{\vec{q_1}}(\bR^n)\times\cdots\times M^{\varphi_m}_{\vec{q_m}}(\bR^n)$ to the
generalized mixed Morrey space $M^\varphi_{\vec{q}}(\bR^n)$.

We first prove the Guliyev local estimate (see, for example, [\cite{guliyev2011boundedness,guliyev2013global}] in the
case $m=1$ and [\cite{guliyev2015multilinear,ismayilovamulti2020,YU2014Boundedness}] in the case $m>1$), which gives us an explicit estimate for the $L^{\vec{q}}(\bR^n)$ norm of  $T_m$ on a given ball $B(x_0,r)$. Here and in what follows, we use the notation $\vec{q}=(q_1,\cdots,q_n)$ and
$\vec{q_i}=(q_{i1},\cdots,q_{in})$ for all $i=1,\cdots,m$.
\begin{lemma}\label{L-T1}
	Let $m\geq2$, $1<\vec{q_i}<\infty$ for all $i=1,\cdots,m$ and $1/\vec{q}=1/\vec{q_1}+\cdots+1/\vec{q_m}$. If $T_m$ is a multi-sublinear operator satisfying condition (\ref{D-multisub}), and bounded from  $L^{\vec{q_1}}(\bR^n)\times\cdots\times L^{\vec{q_m}}(\bR^n)$ to $L^{\vec{q}}(\bR^n)$, then the inequality 
	\begin{equation}\label{E-ball-1}
	\|T_m(\vec{f})\|_{L^{\vec{q}}(B(x_0,r))}\lesssim r^{\sum_{i=1}^n\frac{1}{q_i}}\prod_{i=1}^m\int_{2r}^\infty t^{-1-\sum_{j=1}^n\frac{1}{q_{ij}}}
	\|f_i\|_{L^{\vec{q_i}}(B(x_0,t))}dt
	\end{equation}
	holds for any ball $B(x_0,r)$ and $\vec{f}\in L^{\vec{q_1}}_{\rm loc}(\bR^n)\times\cdots\times L^{\vec{q_m}}_{\rm loc}(\bR^n)$.
\end{lemma}
\begin{proof}
	For any ball $B=B(x_0,r)$, let $2B=B(x_0,2r)$ be the ball centered at $x_0$ with radius $2r$. We represent $f_i$ as $f_i=f^0_i+f^{\infty}_i$, where
	$$f^0_i=f_i\chi_{2B},~~f^{\infty}_i=f_i\chi_{(2B)^c},~~i=1,\cdots,m.$$
	Since $T_m$ is a multi-sublinear operator, we can split $T_m\vec{f}$ as 
	$$|T_m(\vec{f})(x)|\lesssim|T_m(f^0_1,\cdots,f^0_m)(x)|+\l|\sum_{\beta_1,\cdots,\beta_m}
	T_m(f^{\beta_1}_1,\cdots,f^{\beta_m}_m)(x)\r|,$$
	where $\beta_1,\cdots,\beta_m\in\{0,\infty\}$ and each term of $\sum$ contains at least $\beta_i\neq0$. Then
	\begin{eqnarray*}
		\|T_m(\vec{f})\|_{L^{\vec{q}}(B)}
		&\lesssim&\|T_m(f^0_1,\cdots,f^0_m)\|_{L^{\vec{q}}(B)}+\l\|\sum_{\beta_1,\cdots,\beta_m}
		T_m(f^{\beta_1}_1,\cdots,f^{\beta_m}_m)\r\|_{L^{\vec{q}}(B)}\\
		&:=& I+II.
	\end{eqnarray*}
	Denote by $\vec{f^0}=(f^0_1,\cdots,f^0_m)$. For the term $I$, by using the boundedness of $T_m$ from product Lebesgue spaces $L^{\vec{q_1}}(\bR^n)\times\cdots\times L^{\vec{q_m}}(\bR^n)$ to $L^{\vec{q}}(\bR^n)$ with $1/\vec{q}=1/\vec{q_1}+\cdots+1/\vec{q_m}$ and $1<\vec{q_i}<\infty$, we have
	\begin{eqnarray*}
		I&=&\|T_m(\vec{f^0})\|_{L^{\vec{q}}(B)}\leq\|T_m(\vec{f^0})\|_{L^{\vec{q}}(\bR^n)}\\
		&\lesssim&\prod_{i=1}^m\|{f^0_i}\|_{L^{\vec{q_i}}(\bR^n)}\leq\prod_{i=1}^m\|{f_i}\|_{L^{\vec{q_i}}(2B)}.
	\end{eqnarray*}
	On the other hand, for any $i=1,\cdots,m$,
	\begin{eqnarray}\label{E-0}
	\|f_i\|_{L^{\vec{q_i}}(2B)}&\sim& r^{\sum_{j=1}^n\frac{1}{q_{ij}}}
	\|f_i\|_{L^{\vec{q_i}}(2B)}\int_{2r}^\infty\frac{dt}{t^{1+\sum_{j=1}^n\frac{1}{q_{ij}}}}\nonumber\\
	&\lesssim& r^{\sum_{j=1}^n\frac{1}{q_{ij}}}\int_{2r}^\infty\|f_i\|_{L^{\vec{q_i}}(B(x_0,t))}\frac{dt}{t^{1+\sum_{j=1}^n\frac{1}{q_{ij}}}}.
	\end{eqnarray}
	Therefore,
	$$I\lesssim r^{\sum_{i=1}^n\frac{1}{q_i}}\prod_{i=1}^m\int_{2r}^\infty t^{-1-\sum_{j=1}^n\frac{1}{q_{ij}}}
	\|f_i\|_{L^{\vec{q_i}}(B(x_0,t))}dt.$$
	Now we turn to the estimates of $II$. We first consider the case $\beta_1=\cdots=\beta_m=\infty$.
	
	When $x\in B$, $y_i\in (2B)^c~(i=1,\cdots,m)$, we have $\frac{1}{2}|x_0-y_i|\leq |x-y_i|\leq \frac{3}{2}|x_0-y_i|$. Denote by $\vec{f^{\infty}}=(f^{\infty}_1,\cdots,f^{\infty}_m)$. By virtue of condition (\ref{D-multisub}), there holds
	\begin{eqnarray*}
		|T_m(\vec{f^{\infty}})(x)|&\lesssim&\int_{((2B)^c)^m}\frac{|f_1(y_1)\times\cdots\times f_m(y_m)|}{|(x-y_1,\cdots,x-y_m)|^{mn}}d\vec{y}\\
		&\lesssim&
		\int_{((2B)^c)^m}\frac{|f_1(y_1)\times\cdots\times f_m(y_m)|}{|(x_0-y_1,\cdots,x_0-y_m)|^{mn}}d\vec{y}\\
		&\lesssim&\prod_{i=1}^m\int_{(2B)^c}\frac{|f_i(y_i)|}{|x_0-y_i|^n}dy_i.
	\end{eqnarray*}	
	As a consequence,
	\begin{eqnarray*}
		\|T_m(\vec{f^{\infty}})\|_{L^{\vec{q}}(B)}
		&\lesssim&\|\chi_B\|_{L^{\vec{q}}(\bR^n)}\prod_{i=1}^m\int_{(2B)^c}\frac{|f_i(y_i)|}{|x_0-y_i|^n}dy_i\\
		&\lesssim&r^{\sum_{i=1}^n\frac{1}{q_i}}\times\prod_{i=1}^m\int_{(2B)^c}\frac{|f_i(y_i)|}{|x_0-y_i|^n}dy_i.
	\end{eqnarray*}	
	For any $i=1,\cdots,m$, by Fubini's theorem, we have
	\begin{eqnarray*}
		\int_{(2B)^c}\frac{|f_i(y_i)|}{|x_0-y_i|^n}dy_i&\sim& \int_{(2B)^c}|f_i(y_i)|\int_{|x_0-y_i|}^\infty\frac{dt}{t^{n+1}}dy_i\\
		&\sim& \int_{2r}^\infty\int_{2r\leq |x_0-y_i|<t}|f_i(y_i)|dy_i\frac{dt}{t^{n+1}}\\
		&\lesssim& \int_{2r}^\infty\int_{B(x_0,t)}|f_i(y_i)|dy_i\frac{dt}{t^{n+1}}.
	\end{eqnarray*}
	Applying H{\"o}lder's inequality on mixed Lebesgue spaces (see [\cite{1961The}]), we obtain
	\begin{equation}\label{E-a}
	\int_{(2B)^c}\frac{|f_i(y_i)|}{|x_0-y_i|^n}dy\lesssim\int_{2r}^\infty\|f_i\|_{L^{\vec{q_i}}(B(x_0,t))}\frac{dt}{t^{1+\sum_{j=1}^n\frac{1}{q_{ij}}}}.
	\end{equation}
	Hence, we have
	\begin{eqnarray*}
		\|T_m(\vec{f^{\infty}})\|_{L^{\vec{q}}(B)}
		\lesssim r^{\sum_{i=1}^n\frac{1}{q_i}}\prod_{i=1}^m\int_{2r}^\infty t^{-1-\sum_{j=1}^n\frac{1}{q_{ij}}}
		\|f_i\|_{L^{\vec{q_i}}(B(x_0,t))}dt.
	\end{eqnarray*}	
	Next we consider the case that some $\beta_i=0$ and other $\beta_j=\infty$. To this end we may
	assume that $\beta_1=\beta_2=\infty$ and $\beta_3=\cdots=\beta_m=0$. Observing condition (\ref{D-multisub}) and the fact $|x-y_i|\sim|x_0-y_i|$ for all $x\in B$ and $y_i\in (2B)^c, i=1,2$, we arrive at
	\begin{eqnarray*}
		&&|T_m(f_1^{\infty},f_2^{\infty},f^0_3,\cdots,f^0_m)(x)|\\
		&\lesssim&\int_{(2B)^c\times(2B)^c}\frac{|f_1(y_1)\times f_2(y_2)|}{(|x-y_1|+|x-y_2|)^{mn}}dy_1dy_2\times\prod_{i=3}^m\int_{2B}|f_i(y_i)|dy_i\\
		&\lesssim&r^{-n(m-2)}\int_{(2B)^c}\frac{|f_1(y_1)|}{|x_0-y_1|^n}dy_1\times
		\int_{(2B)^c}\frac{|f_2(y_2)|}{|x_0-y_2|^n}dy_2\times\prod_{i=3}^m\int_{2B}|f_i(y_i)|dy_i.
	\end{eqnarray*}
	By using inequalities (\ref{E-0}), (\ref{E-a}) and H{\"o}lder's inequality on mixed Lebesgue spaces, we get
	\begin{eqnarray*}
		&&\|T_m(f_1^{\infty},f_2^{\infty},f^0_3,\cdots,f^0_m)\|_{L^{\vec{q}}(B)}\\
		&\lesssim&r^{\sum_{i=1}^n\frac{1}{q_i}}\times r^{-n(m-2)}\int_{(2B)^c}\frac{|f_1(y_1)|}{|x_0-y_1|^n}dy_1\times
		\int_{(2B)^c}\frac{|f_2(y_2)|}{|x_0-y_2|^n}dy_2\times\prod_{i=3}^m\int_{2B}|f_i(y_i)|dy_i\\
		&\lesssim&r^{\sum_{i=1}^n\frac{1}{q_i}}\times r^{-n(m-2)} \int_{(2B)^c}\frac{|f_1(y_1)|}{|x_0-y_1|^n}dy_1\times
		\int_{(2B)^c}\frac{|f_2(y_2)|}{|x_0-y_2|^n}dy_2\times\prod_{i=3}^mr^{\sum_{j=1}^n\frac{1}{q'_{ij}}}\|f_i\|_{L^{\vec{q_i}}(2B)}\\
		&\lesssim& r^{\sum_{i=1}^n\frac{1}{q_i}}\prod_{i=1}^m\int_{2r}^\infty t^{-1-\sum_{j=1}^n\frac{1}{q_{ij}}}
		\|f_i\|_{L^{\vec{q_i}}(B(x_0,t))}dt.
	\end{eqnarray*}
	From the estimates of $I$ and $II$, we obtain the desired result (\ref{E-ball-1}).
\end{proof}
Now we give the boundedness of the multi-sublinear operators generated by multilinear Calder{\'o}n-Zygmund operators on product generalized mixed Morrey spaces.
\begin{theorem}\label{T-T1}
	Let $m\geq2$, $1<\vec{q_i}<\infty$ for all $i=1,\cdots,m$, $1/\vec{q}=1/\vec{q_1}+\cdots+1/\vec{q_m}$, and $\varphi,\varphi_i~(i=1,\cdots,m): \bR^n\times (0,\infty)\rightarrow(0,\infty)$ be Lebesgue measurable functions satisfying
	\begin{equation}\label{condition-1}
		\prod_{i=1}^m\int_r^\infty\frac{\varphi_i(x,t)}{t}dt\lesssim \varphi(x,r).
		\end{equation} 
	If $T_m$ is a multi-sublinear operator satisfying condition (\ref{D-multisub}), and bounded from  $L^{\vec{q_1}}(\bR^n)\times\cdots\times L^{\vec{q_m}}(\bR^n)$ to $L^{\vec{q}}(\bR^n)$, then $T_m$ is also bounded from product spaces $M^{\varphi_1}_{\vec{q_1}}(\bR^n)\times\cdots\times M^{\varphi_m}_{\vec{q_m}}(\bR^n)$ to
	$M^{\varphi}_{\vec{q}}(\bR^n)$.
\end{theorem} 
\begin{proof}
	By using (\ref{condition-1}) and Lemma \ref{L-T1}, we have
	\begin{eqnarray*}
		\|T_m(\vec{f})\|_{M^{\varphi}_{\vec{q}}}&\lesssim& \sup_{x\in\bR^n,r>0}\varphi(x,r)^{-1}\prod_{i=1}^m\int_{r}^\infty t^{-1-\sum_{j=1}^n\frac{1}{q_{ij}}}
		\|f_i\|_{L^{\vec{q_i}}(B(x,t))}dt\\
		&=& \sup_{x\in\bR^n,r>0}\varphi(x,r)^{-1}\prod_{i=1}^m\int_{r}^\infty \frac{\varphi_i(x,t)}{t}
		\varphi_i(x,t)^{-1}\|\chi_{B(x,t)}\|_{L^{\vec{q_i}}}^{-1}\|f_i\|_{L^{\vec{q_i}}(B(x,t))}dt\\
		&\lesssim& \sup_{x\in\bR^n,r>0}\varphi(x,r)^{-1}\prod_{i=1}^m\int_{r}^\infty \frac{\varphi_i(x,t)}{t}
		dt\times\|f_i\|_{M^{\varphi_k}_{\vec{q_i}}}\lesssim\prod_{i=1}^m\|f_i\|_{M^{\varphi_k}_{\vec{q_i}}}.
	\end{eqnarray*} 
	The proof is complete.
\end{proof}
By taking $\vec{q_i}=q_i$ and $\vec{q}=q$ in Theorem \ref{T-T1}, we recover the results of Guliyev et al. [\cite{guliyev2015multilinear}, Theorem 1.1],
Ismayilova et al. [\cite{ismayilovamulti2020}, Theorem 2.3] and Yu et al. [\cite{YU2014Boundedness}, Theorem 2.1 and Theorem 3.1], which established the boundedness of $T_m$ on product generalized Morrey spaces.

\section{Boundeness of $T^b_{m,i}$ from 
	$M^{\varphi_1}_{\vec{q_1}}(\bR^n)\times\cdots\times M^{\varphi_m}_{\vec{q_m}}(\bR^n)$ to $M^{\varphi}_{\vec{q}}(\bR^n)$ }
The aim of this section is to establish the boundedness of $T^b_{m,i}$ satisfying size condition (\ref{D-multisub-commu}) from product generalized mixed Morrey spaces
$M^{\varphi_1}_{\vec{q_1}}(\bR^n)\times\cdots\times M^{\varphi_m}_{\vec{q_m}}(\bR^n)$ to the
generalized mixed Morrey space $M^\varphi_{\vec{q}}(\bR^n)$, where the symbol function $b$ is always assumed to be in BMO.


As in Section 2, we first prove the Guliyev local estimate for $T^b_{m,i}$.
\begin{lemma}\label{L-T3}
	Let $m\geq2$, $1<\vec{q_k}<\infty$ for $k=1,\cdots,m$ and $1/\vec{q}=1/\vec{q_1}+\cdots+1/\vec{q_m}$. If 
	$b\in {\rm BMO}$, $T^b_{m,i}$ is a multi-sublinear operator satisfying condition (\ref{D-multisub-commu}), and bounded from  $L^{\vec{q_1}}(\bR^n)\times\cdots\times L^{\vec{q_m}}(\bR^n)$ to $L^{\vec{q}}(\bR^n)$, then the inequality 
	\begin{eqnarray}\label{E-ball-2}
	\|T^b_{m,i}(\vec{f})\|_{L^{\vec{q}}(B(x_0,r))}&\lesssim& r^{\sum_{j=1}^n\frac{1}{q_j}}\int_{2r}^\infty \l(1+\ln\frac{t}{r}\r)t^{-1-\sum_{j=1}^n\frac{1}{q_{ij}}}
	\|f_i\|_{L^{\vec{q_i}}(B(x_0,t))}dt \nonumber\\
	&\times&\prod_{k=1,k\neq i}^m\int_{2r}^\infty t^{-1-\sum_{j=1}^n\frac{1}{q_{kj}}}
	\|f_k\|_{L^{\vec{q_k}}(B(x_0,t))}dt
	\end{eqnarray}
	holds for any ball $B(x_0,r)$ and $\vec{f}\in L^{\vec{q_1}}_{\rm loc}(\bR^n)\times\cdots\times L^{\vec{q_m}}_{\rm loc}(\bR^n)$.
\end{lemma}
\begin{proof}
	As in Lemma \ref{L-T1}, for any ball $B=B(x_0,r)$, let $2B=B(x_0,2r)$ be the ball centered at $x_0$ with radius $2r$. We still use the representation $f_k=f^0_k+f^{\infty}_k$, where
	$$f^0_k=f_k\chi_{2B},~~f^{\infty}_k=f_k\chi_{(2B)^c},~~k=1,\cdots,m.$$
	Since $T^b_{m,i}$ is a multi-sublinear operator, we can split $T^b_{m,i}\vec{f}$ as 
	$$|T^b_{m,i}(\vec{f})(x)|\lesssim|T^b_{m,i}(f^0_1,\cdots,f^0_m)(x)|+\l|\sum_{\beta_1,\cdots,\beta_m}
	T^b_{m,i}(f^{\beta_1}_1,\cdots,f^{\beta_m}_m)(x)\r|,$$
	where $\beta_1,\cdots,\beta_m\in\{0,\infty\}$ and each term of $\sum$ contains at least $\beta_i\neq0$.
	Then
	\begin{eqnarray*}
		\|T^b_{m,i}(\vec{f})\|_{L^{\vec{q}}(B)}
		&\lesssim&\|T^b_{m,i}(f^0_1,\cdots,f^0_m)\|_{L^{\vec{q}}(B)}+\l\|\sum_{\beta_1,\cdots,\beta_m}
		T^b_{m,i}(f^{\beta_1}_1,\cdots,f^{\beta_m}_m)\r\|_{L^{\vec{q}}(B)}\\
		&:=& I+II.
	\end{eqnarray*}
	Denote by $\vec{f^0}=(f^0_1,\cdots,f^0_m)$. For the term $I$, by using the boundedness of $T^b_{m,i}$, we obtain
	\begin{eqnarray*}
		I&=&\|T^b_{m,i}(\vec{f^0})\|_{L^{\vec{q}}(B)}\leq\|T^b_{m,i}(\vec{f^0})\|_{L^{\vec{q}}(\bR^n)}\\
		&\lesssim&\prod_{k=1}^m\|{f^0_k}\|_{L^{\vec{q_k}}(\bR^n)}\leq\prod_{k=1}^m\|{f_k}\|
		_{L^{\vec{q_k}}(2B)}.
	\end{eqnarray*}
	Therefore, we get from (\ref{E-0}) that
	$$I\lesssim r^{\sum_{j=1}^n\frac{1}{q_j}}\prod_{k=1}^m\int_{2r}^\infty t^{-1-\sum_{j=1}^n\frac{1}{q_{kj}}}
	\|f_k\|_{L^{\vec{q_k}}(B(x_0,t))}dt.$$
	Next we will give the estimates of $II$. We first consider the case $\beta_1=\cdots=\beta_m=\infty$.
	For $x\in B$, $y_k\in (2B)^c~(k=1,\cdots,m)$, we have $\frac{1}{2}|x_0-y_k|\leq |x-y_k|\leq \frac{3}{2}|x_0-y_k|$. Denote by $\vec{f^{\infty}}=(f^{\infty}_1,\cdots,f^{\infty}_m)$. By using condition (\ref{D-multisub-commu}), we obtain
	\begin{eqnarray*}
		|T^b_{m,i}(\vec{f^{\infty}})(x)|&\lesssim&\int_{((2B)^c)^m}\frac{|b(x)-b(y_i)|}{|(x-y_1,\cdots,x-y_m)|^{mn}}
		|f_1(y_1)\times\cdots\times f_m(y_m)|d\vec{y}\nonumber\\
		&\lesssim&
		\int_{((2B)^c)^m}\frac{|b(x)-b(y_i)|}{|(x_0-y_1,\cdots,x_0-y_m)|^{mn}}|f_1(y_1)\times\cdots\times f_m(y_m)|d\vec{y}\nonumber\\
		&\lesssim&\int_{(2B)^c}\frac{|b(x)-b(y_i)|}{|x_0-y_i|^n}f_i(y_i)dy_i\times\prod_{k=1,k\neq i}^m\int_{(2B)^c}\frac{|f_k(y_k)|}{|x_0-y_k|^n}dy_k.
	\end{eqnarray*}	
	By using H{\"o}lder's inequality on mixed Lebesgue spaces, we have 
	\begin{eqnarray}\label{E-zong}
	\|T^b_{m,i}(\vec{f^{\infty}})\|_{L^{\vec{q}}(B)}&\lesssim&\l\|\int_{(2B)^c}\frac{|b(x)-b(y_i)|}{|x_0-y_i|^n}f_i(y_i)dy_i\r\|_{L^{\vec{q_i}}(B)}\nonumber\\
	&\times&\prod_{k=1,k\neq i}^m\|\chi_B\|_{L^{\vec{q_k}}(B)}\int_{(2B)^c}\frac{|f_k(y_k)|}{|x_0-y_k|^n}dy_k.
	\end{eqnarray}
	
By using an equivalent characterization of BMO in terms of mixed norm ([\cite{ho2018mixed}]) and the trival inequality $|f_{B(x,r)}-f_{B(x,t)}|\lesssim \|f\|_{{\rm BMO}}\ln\frac{t}{r}$ for $f\in {\rm BMO}(\bR^n)$ and $0<2r<t$, Wei [\cite{wei2021boundedness}] proved that
	\begin{eqnarray}\label{E-21}
	&&\l\|\int_{(2B)^c}\frac{|b(x)-b(y_i)|}{|x_0-y_i|^n}f_i(y_i)dy_i\r\|_{L^{\vec{q_i}}(B)}\nonumber\\
	&\lesssim& \|b\|_{\rm BMO}r^{\sum_{j=1}^n\frac{1}{q_{ij}}}
	\int_{2r}^\infty \l(1+\ln\frac{t}{r}\r)t^{-1-\sum_{j=1}^n\frac{1}{q_{ij}}}\|f_i\|_{L^{\vec{q_i}}(B(x_0,t))}dt,
	\end{eqnarray}
see [\cite{wei2021boundedness}, Proof of Lemma 4.2] for the details.
	
	Inserting (\ref{E-a}) and (\ref{E-21}) into (\ref{E-zong}), we arrive at
	\begin{eqnarray*}
		\|T^b_{m,i}(\vec{f^{\infty}})\|_{L^{\vec{q}}(B)}&\lesssim&
		\|b\|_{\rm BMO}r^{\sum_{j=1}^n\frac{1}{q_j}}\int_{2r}^\infty \l(1+\ln\frac{t}{r}\r)t^{-1-\sum_{j=1}^n\frac{1}{q_{ij}}}
		\|f_i\|_{L^{\vec{q_i}}(B(x_0,t))}dt \nonumber\\
		&\times&\prod_{k=1,k\neq i}^m\int_{2r}^\infty t^{-1-\sum_{j=1}^n\frac{1}{q_{kj}}}
		\|f_k\|_{L^{\vec{q_k}}(B(x_0,t))}dt.
	\end{eqnarray*}
	Now we consider the case that some $\beta_k=0$ and other $\beta_j=\infty$. Without loss of generality, we assume $\beta_1=\beta_2=\infty$ and $\beta_3=\cdots=\beta_m=0$ and $i=2$. In view of condition (\ref{D-multisub-commu}) and the fact $|x-y_k|\sim|x_0-y_k|$ for all $x\in B$ and $y_k\in (2B)^c, k=1,2$, we obtain
	\begin{eqnarray*}
		&&|T^b_{m,i}(f_1^{\infty},f_2^{\infty},f^0_3,\cdots,f^0_m)(x)|\\
		&\lesssim&\int_{(2B)^c\times(2B)^c}\frac{|b(x)-b(y_2)|}{(|x-y_1|+|x-y_2|)^{mn}}|f_1(y_1)\times f_2(y_2)|dy_1dy_2\times\prod_{k=3}^m\int_{2B}|f_k(y_k)|dy_k\\
		&\lesssim&r^{-n(m-2)}\int_{(2B)^c}\frac{|f_1(y_1)|}{|x_0-y_1|^n}dy_1\times
		\int_{(2B)^c}\frac{|b(x)-b(y_2)|}{|x_0-y_2|^n}|f_2(y_2)|dy_2\times\prod_{k=3}^m\int_{2B}|f_k(y_k)|dy_k.
	\end{eqnarray*}
	Combining (\ref{E-0}), (\ref{E-a}) with (\ref{E-21}) and then using H{\"o}lder's inequality on mixed Lebesgue spaces, we get
	\begin{eqnarray*}
		&&\|T^b_{m,i}(f_1^{\infty},f_2^{\infty},f^0_3,\cdots,f^0_m)\|_{L^{\vec{q}}(B)}\\
		&\lesssim&r^{\sum_{j=1}^n\frac{1}{q_j}}\times r^{-n(m-2)}\int_{(2B)^c}\frac{|f_1(y_1)|}{|x_0-y_1|^n}dy_1\\
		&\times&
		\int_{(2B)^c}\frac{|b(x)-b(y_2)|}{|x_0-y_2|^n}|f_2(y_2)|dy_2\times\prod_{k=3}^m\int_{2B}|f_k(y_k)|dy_k\\
		&\lesssim&r^{\sum_{j=1}^n\frac{1}{q_j}}\times \int_{(2B)^c}\frac{|f_1(y_1)|}{|x_0-y_1|^n}dy_1\\
		&\times&
		\int_{(2B)^c}\frac{|b(x)-b(y_2)|}{|x_0-y_2|^n}|f_2(y_2)|dy_2\times\prod_{k=3}^mr^{-\sum_{j=1}^n\frac{1}{q_{kj}}}\|f_k\|_{L^{\vec{q_k}}(2B)}\\
		&\lesssim&
		\|b\|_{\rm BMO}r^{\sum_{j=1}^n\frac{1}{q_j}}\int_{2r}^\infty \l(1+\ln\frac{t}{r}\r)t^{-1-\sum_{j=1}^n\frac{1}{q_{ij}}}
		\|f_i\|_{L^{\vec{q_i}}(B(x_0,t))}dt \nonumber\\
		&\times&\prod_{k=1,k\neq i}^m\int_{2r}^\infty t^{-1-\sum_{j=1}^n\frac{1}{q_{kj}}}
		\|f_k\|_{L^{\vec{q_k}}(B(x_0,t))}dt.
	\end{eqnarray*}
	From the estimates of $I$ and $II$, we get the desired result (\ref{E-ball-2}).
\end{proof}

Now we give the boundedness of $T^b_{m,i}$ on product generalized mixed Morrey spaces.
\begin{theorem}\label{T-T3}
	Let $m\geq2$, $1<\vec{q_k}<\infty$ for all $k=1,\cdots,m$, $1/\vec{q}=1/\vec{q_1}+\cdots+1/\vec{q_m}$, and $\varphi,\varphi_k~(k=1,\cdots,m): \bR^n\times (0,\infty)\rightarrow(0,\infty)$ be Lebesgue measurable functions satisfying
 \begin{eqnarray}\label{condition-2}
		&&\int_r^\infty\l(1+\ln{\frac{t}{r}}\r)\frac{\varphi_i(x,t)}{t}dt
		\times\prod_{k=1,k\neq i}^m\int_r^\infty\frac{\varphi_k(x,t)}{t}dt\lesssim \varphi(x,r).
		\end{eqnarray} 
	If $b\in{\rm BMO}$, $T^b_{m,i}$ is a multi-sublinear operator satisfying condition (\ref{D-multisub-commu}), and bounded from  $L^{\vec{q_1}}(\bR^n)\times\cdots\times L^{\vec{q_m}}(\bR^n)$ to $L^{\vec{q}}(\bR^n)$, then $T^b_{m,i}$ is also bounded from $M^{\varphi_1}_{\vec{q_1}}(\bR^n)\times\cdots\times M^{\varphi_m}_{\vec{q_m}}(\bR^n)$ to
	$M^{\varphi}_{\vec{q}}(\bR^n)$.
\end{theorem} 
\begin{proof}
By using Lemma \ref{L-T3} and (\ref{condition-2}), we have
	\begin{eqnarray*}
		\|T^b_{m,i}(\vec{f})\|_{M^{\varphi}_{\vec{q}}}&\lesssim& \sup_{x\in\bR^n,r>0}\varphi(x,r)^{-1}\int_{r}^\infty \l(1+\ln\frac{t}{r}\r)t^{-1-\sum_{j=1}^n\frac{1}{q_{ij}}}
		\|f_i\|_{L^{\vec{q_i}}(B(x_0,t))}dt\nonumber\\
		&\times&\prod_{k=1,k\neq i}^m\int_{r}^\infty t^{-1-\sum_{j=1}^n\frac{1}{q_{kj}}}
		\|f_k\|_{L^{\vec{q_k}}(B(x_0,t))}dt \\
		&=& \sup_{x\in\bR^n,r>0}\varphi(x,r)^{-1}\int_{r}^\infty \l(1+\ln\frac{t}{r}\r)\frac{\varphi_i(x,t)}{t}
		\varphi_i(x,t)^{-1}\|\chi_{B(x,t)}\|_{L^{\vec{q_i}}}^{-1}\|f_i\|_{L^{\vec{q_i}}(B(x_0,t))}dt\nonumber\\
		&\times&\prod_{k=1,k\neq i}^m\int_{r}^\infty \frac{\varphi_k(x,t)}{t}
		\varphi_k(x,t)^{-1}\|\chi_{B(x,t)}\|_{L^{\vec{q_k}}}^{-1}\|f_k\|_{L^{\vec{q_k}}(B(x_0,t))}dt \\
		&\lesssim& 
		\prod_{k=1}^m\|f_k\|_{M^{\varphi_k}_{\vec{q_k}}}.
	\end{eqnarray*}
	We are done.
\end{proof}

\section{Some applications}
In this section, we give some applications of our main theorems. We will show that many multi-sublinear operators and their commutators satisfy the assumptions in Theorem \ref{T-T1} and Theorem \ref{T-T3}. Therefore, we can obtain the boundedness of these operators on product generalized mixed Morrey spaces by using our main results.

To prove the boundedness of $M_{m}$, $M^{\vec{b}}_{m}$  and $M^{\vec{b}}_{m,i}$ from $L^{\vec{q_1}}(\bR^n)\times\cdots\times L^{\vec{q_m}}(\bR^n)$ to $L^{\vec{q}}(\bR^n)$, we need the boundedness of $M$ and $M^b$ on mixed norm space $L^{\vec{q}}(\bR^n)$.
\begin{lemma}([\cite{nogayama2019mixed}])\label{l0-A}
	Let $1<\vec{q}<\infty$. Then
	\begin{equation}\label{0-A}
	\|M(f)\|_{L^{\vec{q}}}\lesssim\|f\|_{L^{\vec{q}}}.
	\end{equation}
\end{lemma}
\begin{lemma}([\cite{nogayama2019mixed}])\label{l0-B}
	Let $1<\vec{q}<\infty$ and $b\in {\rm BMO}$. Then
	\begin{equation}\label{0-B}
	\|M^b(f)\|_{L^{\vec{q}}}\lesssim\|b\|_{\rm BMO}\|f\|_{L^{\vec{q}}}.
	\end{equation}
\end{lemma}
Now we give the boundedness of multi-sublinear maximal operator and its commutator on product generalized mixed Morrey spaces.
\begin{theorem}\label{App-T-1}
	Let $m\geq2$, $1<\vec{q_i}<\infty$ for all $i=1,\cdots,m$ and $1/\vec{q}=1/\vec{q_1}+\cdots+1/\vec{q_m}$. Then $M_{m}$ is
	bounded from $L^{\vec{q_1}}(\bR^n)\times\cdots\times L^{\vec{q_m}}(\bR^n)$ to $L^{\vec{q}}(\bR^n)$.
\end{theorem}
\begin{proof}
	The simple inequality
	\begin{eqnarray*}
		M_{m}(\vec{f})(x)\leq\prod_{i=1}^m M(f_i)(x)
	\end{eqnarray*}
	together with Lemma \ref{l0-A} and H{\"o}lder's inequality on mixed Lebesgue spaces gives us that for all $\vec{f}=(f_1,\cdots,f_m)\in L^{\vec{q_1}}(\bR^n)\times\cdots\times L^{\vec{q_m}}(\bR^n)$,
	\begin{eqnarray*}
		\|M_{m}(\vec{f})\|_{L^{\vec{q}}}\leq \prod_{i=1}^m \|M(f_i)\|_{L^{\vec{q_i}}}\lesssim \prod_{i=1}^m\|f_i\|_{L^{\vec{q_i}}}.
	\end{eqnarray*}
\end{proof}
\begin{theorem}\label{App-T-2}
	Let $m\geq2$, $1<\vec{q_k}<\infty$ for all $k=1,\cdots,m$ and $1/\vec{q}=1/\vec{q_1}+\cdots+1/\vec{q_m}$. If $\vec{b}=(b_1,\cdots,b_m)\in{\rm BMO}^m$, then both $M^{\vec{b}}_{m}$ and $M^{\vec{b}}_{m,i}$ are 
	bounded from $L^{\vec{q_1}}(\bR^n)\times\cdots\times L^{\vec{q_m}}(\bR^n)$ to $L^{\vec{q}}(\bR^n)$.
\end{theorem}
\begin{proof}
	We only need to prove the boundedness of $M^{\vec{b}}_{m,i}$, and the boundedness of $M^{\vec{b}}_{m}$ follows directly. Obviously,
	\begin{eqnarray}\label{G-4_2}
	M^{\vec{b}}_{m,i}(\vec{f})(x)\leq M^{b_i}(f_i)(x)\times\prod_{j=1,j\neq i}^m M(f_j)(x).
	\end{eqnarray}
	Inequality (\ref{G-4_2}) together with Lemma \ref{l0-A}, Lemma \ref{l0-B} and H{\"o}lder's inequality on mixed Lebesgue spaces implies that for all $\vec{f}=(f_1,\cdots,f_m)\in L^{\vec{q_1}}(\bR^n)\times\cdots\times L^{\vec{q_m}}(\bR^n)$,
	\begin{eqnarray*}
		\|M^{\vec{b}}_{m,i}(\vec{f})\|_{L^{\vec{q}}}\leq \|M^{b_i}(f_i)\|_{L^{\vec{q_i}}}\times\prod_{j=1,j\neq i}^m \|M(f_j)\|_{L^{\vec{q_j}}}\lesssim \|b\|_{\rm BMO}\prod_{j=1}^m\|f_j\|_{L^{\vec{q_j}}}.
	\end{eqnarray*}
\end{proof}

To verify the boundedness of $K_{m}$, $K^{\vec{b}}_{m}$ and $K^{\vec{b}}_{m,i}$ from  $L^{\vec{q_1}}(\bR^n)\times\cdots\times L^{\vec{q_m}}(\bR^n)$ to $L^{\vec{q}}(\bR^n)$, we need the extrapolation theory on mixed Lebesgue spaces.
The extrapolaton theory on mixed Lebesgue spaces relies on the classical $A_p$ weight (see [\cite{grafakos2008classical}]). 
\begin{definition}
	For $1<p<\infty$, a locally integrable function $w: \bR^n\rightarrow (0,\infty)$ is said to be an $A_p$ weight if
	\begin{eqnarray*}
		[w]_{A_p}=\sup_{B}\l(\frac{1}{|B|}\int_Bw(x)dx\r)\l(\frac{1}{|B|}\int_Bw(x)^{-\frac{p'}{p}}dx\r)^{\frac{p}{p'}}<\infty,
	\end{eqnarray*}
	where $B$ is taken over all the balls in $\bR^n$.
	
	A a locally integrable function $w: \bR^n\rightarrow (0,\infty)$ is said to be an $A_1$ weight if for any ball $B\subseteq \bR^n$,
	$$\frac{1}{|B|}\int_Bw(y)dy\leq Cw(x),~~~ a.e. x\in B$$
	for some constant $C>0$. The infimum of all such $C$ is denoted by $[w]_{A_1}$. We denote $A_\infty$ by the union of all $A_p~(1\leq p<\infty)$ functions.
\end{definition}
For a non-negative function $w\in L^1_{\rm{loc}}(\bR^n)$ and $0<p<\infty$, the weighted Lebesgue space $L^p_w(\bR^n)$ consists of all $f\in \mathcal{M}(\bR^n)$ such that $$\|f\|_{L^p_w}:=\left(\int_{\bR^n}|f(x)|^pw(x)dx\right)^{1/p}<\infty.$$

Now we give the extrapolation theorem on mixed Lebesgue spaces. 
\begin{theorem}\label{D-ext}
	Let $0<q_0<\infty$, $\vec{q}=(q_1,\cdots,q_n)\in(0,\infty)^n$, and $\mathfrak{F}$  be a family of ordered pairs of non-negative measurable functions. Suppose for all $w\in A_1$ and $(f,g)\in\mathfrak{F}$, we have
	\begin{equation}\label{3-a}
	\int_{\bR^n}f(x)^{q_0}w(x)dx\lesssim \int_{\bR^n}g(x)^{q_0}w(x)dx.
	\end{equation}
	Then if $q_0<\vec{q}$, the inequality
	\begin{equation}\label{3-b}
	\|f\|_{L^{\vec{q}}}\lesssim \|g\|_{L^{\vec{q}}}
	\end{equation}
	holds for all pairs $(f,g)\in\mathfrak{F}$ such that the left-hand side is finite.
\end{theorem}
We refer the reader to [\cite{wei2021boundedness}] for the proof of Theorem \ref{D-ext}. 

The following two lemmas offer the weighted norm estimates for $K_{m}$, $K^{\vec{b}}_{m}$ and $K^{\vec{b}}_{m,i}$.
\begin{lemma}\label{L-4-1}
	Let $q>0$ and $w$ be a weight in $A_{\infty}$, then there exists a constant $C>0$ (depending the $A_{\infty}$ constant of $w$) so that the inequality
	\begin{equation*}
	\|K_{m}(\vec{f})\|_{L^q_w}\leq C\|M_m(\vec{f})\|_{L^q_w}
	\end{equation*}
	holds for all bounded funtions $\vec{f}$ with compact support.
\end{lemma}
\begin{lemma}\label{L-4-2}
	Let $q>0$ and $w$ be a weight in $A_{\infty}$. Suppose $\vec{b}\in {\rm BMO}^m$. Then there exists a constant $C>0$ (depending the $A_{\infty}$ constant of $w$) so that the inequalitis
	\begin{equation*}
	\|K^{\vec{b}}_{m}(\vec{f})\|_{L^q_w}\leq C\l\|\prod_{j=1}^mM^2(f_j)\r\|_{L^q_w}
	\end{equation*}
	and
	\begin{equation*}
	\|K^{\vec{b}}_{m,i}(\vec{f})\|_{L^q_w}\leq C\l\|\prod_{j=1}^mM^2(f_j)\r\|_{L^q_w}
	\end{equation*}
	hold for all bounded funtions $\vec{f}$ with compact support, where $M^2f_j=M(Mf_j)$.
\end{lemma}
We refer the readers to [\cite{lerner2009new}] for the proof of the above two lemmas. Now we can prove the boundedness of multilinear Calder{\'o}n-Zygmund operators and their commutators on product mixed Lebesgue spaces.
\begin{theorem}\label{T-4-2}
	Let $m\geq2$, $1<\vec{q_i}<\infty$ for all $i=1,\cdots,m$ and $1/\vec{q}=1/\vec{q_1}+\cdots+1/\vec{q_m}$. Then $K_{m}$ is
	bounded from $L^{\vec{q_1}}(\bR^n)\times\cdots\times L^{\vec{q_m}}(\bR^n)$ to $L^{\vec{q}}(\bR^n)$.
\end{theorem}
\begin{proof}
	Let $\vec{f}=(f_1,\cdots,f_m)\in L^{\vec{q_1}}(\bR^n)\times\cdots\times L^{\vec{q_m}}(\bR^n)$. Since the class of bounded functions with compact support
	are dense in $L^{\vec{q_i}}(\bR^n)~(i=1\cdots,m)$ (see [\cite{1961The}]), it suffices to prove Theorem \ref{T-4-2} for all $f_1,\cdots,f_m\in L^{\infty}_c(\bR^n)$, the bounded functions with compact support.
	
	We define a sequence of operators $\{K^k_{m}\}_{k\in\mathbb{N}}$, where $$K^k_{m}(\vec{f})(x)=\min\{K_{m}(\vec{f}),k\}\chi_{B(0,k)}(x),$$
	and consider the following fimily
	$$\mathfrak{F}=\l\{\l(K^k_{m}(\vec{f}),\prod_{i=1}^mM(f_i)\r):~\vec{f}\in(L^{\infty}_c(\bR^n))^m,~k\in\mathbb{N}^+\r\}.$$
	It follows from Lemma \ref{L-4-1} that for any $0<p_0<\infty$ and $w\in A_1\subseteq A_{\infty}$,
	\begin{eqnarray*}
		\|K_{m}(\vec{f})\|_{L^{p_0}_w}\lesssim\|M_m(\vec{f})\|_{L^{p_0}_w}
		\lesssim\|\prod_{i=1}^mM(f_i)\|_{L^{p_0}_w},
	\end{eqnarray*}
	for every pair $\l(K^k_{m}(\vec{f}),\prod_{i=1}^mM(f_i)\r)\in\mathfrak{F}$.
	
	So, to use Theorem \ref{D-ext} for all ordered pairs in $\mathfrak{F}$, we need to check 
	$\|K^k_{m}(\vec{f})\|_{L^{\vec{q}}}<\infty$ for every $\vec{f}\in (L^{\infty}_c(\bR^n))^m$.
	It is always the case since $\|\chi_{B(0,k)}\|_{L^{\vec{q}}}<\infty$.	
	
	Now, we can apply Theorem \ref{D-ext} to each pair in $\mathfrak{F}$ and get		
	\begin{eqnarray*}
		\|K^k_{m}(\vec{f})\|_{L^{\vec{q}}}
		\lesssim\l\|\prod_{i=1}^mM(f_i)\r\|_{L^{\vec{q}}}
	\end{eqnarray*}
	for all $\vec{f}\in (L^{\infty}_c(\bR^n))^m$.
	From the proof of Theorem \ref{App-T-1}, we know 
	\begin{eqnarray*}
		\l\|\prod_{i=1}^m M(f_i)\r\|_{L^{\vec{q}}}\lesssim \prod_{i=1}^m\|f_i\|_{L^{\vec{q_i}}}. 
	\end{eqnarray*}		
	Hence
	\begin{eqnarray*}
		\|K^k_{m}(\vec{f})\|_{L^{\vec{q}}}
		\lesssim \prod_{i=1}^m\|f_i\|_{L^{\vec{q_i}}}
	\end{eqnarray*}
	for all $\vec{f}\in (L^{\infty}_c(\bR^n))^m$.
	Noting that ${K^k_{m}(\vec{f})}(x)$ increases almost everywhere to $K_{m}(\vec{f})(x)$, we get from the monotone convergence theorem on mixed Lebesgue spaces (see [\cite{1961The}]) that 
	\begin{eqnarray*}
		\|K_{m}(\vec{f})\|_{L^{\vec{q}}}=\lim_{k\rightarrow\infty}\|K^k_{m}(\vec{f})\|_{L^{\vec{q}}}
		\lesssim \prod_{i=1}^m\|f_i\|_{L^{\vec{q_i}}}.
	\end{eqnarray*}
\end{proof}
\begin{theorem}\label{T-4-3}
	Let $m\geq2$, $1<\vec{q_k}<\infty$ for all $k=1,\cdots,m$ and $1/\vec{q}=1/\vec{q_1}+\cdots+1/\vec{q_m}$. If $\vec{b}=(b_1,\cdots,b_m)\in {\rm BMO}^m$, then both $K^{\vec{b}}_{m}$ and $K^{\vec{b}}_{m,i}$ are 
	bounded from $L^{\vec{q_1}}(\bR^n)\times\cdots\times L^{\vec{q_m}}(\bR^n)$ to $L^{\vec{q}}(\bR^n)$.
\end{theorem}
\begin{proof}
	In view of Lemma \ref{L-4-2}, the proof is just an iteration of the proof for Theorem \ref{T-4-2}, so we omit the details.
\end{proof}
By using Theorems \ref{T-T1}, \ref{T-T3} and Theorems \ref{App-T-1}, \ref{App-T-2}, we have the boundedness of multi-sublinear maximal operator and its commutators on product generalized mixed Morrey spaces.
\begin{theorem}\label{C-1}
	Let $m\geq2$, $1<\vec{q_i}<\infty$ for all $i=1,\cdots,m$, $1/\vec{q}=1/\vec{q_1}+\cdots+1/\vec{q_m}$, and $\varphi,\varphi_i~(i=1,\cdots,m): \bR^n\times (0,\infty)\rightarrow(0,\infty)$ be Lebesgue measurable functions satisfying condition (\ref{condition-1}).
	Then $M_{m}$ is bounded from product spaces $M^{\varphi_1}_{\vec{q_1}}(\bR^n)\times\cdots\times M^{\varphi_m}_{\vec{q_m}}(\bR^n)$ to
	$M^{\varphi}_{\vec{q}}(\bR^n)$.
\end{theorem}
\begin{proof}
	From Theorem \ref{App-T-1}, we can get the boundedness of $M_{m}$ from $L^{\vec{q_1}}(\bR^n)\times\cdots\times L^{\vec{q_m}}(\bR^n)$ to $L^{\vec{q}}(\bR^n)$. Therefore, the conditions of Theorem \ref{T-T1} are fulfilled for $M_{m}$. As a consequence of Theorem \ref{T-T1}, $M_{m}$ is bounded from $M^{\varphi_1}_{\vec{q_1}}(\bR^n)\times\cdots\times M^{\varphi_m}_{\vec{q_m}}(\bR^n)$ to
	$M^{\varphi}_{\vec{q}}(\bR^n)$.
\end{proof}
\begin{theorem}\label{C-2}
	Let $m\geq2$, $1<\vec{q_k}<\infty$ for all $k=1,\cdots,m$, $1/\vec{q}=1/\vec{q_1}+\cdots+1/\vec{q_m}$, and $\varphi,\varphi_k~(k=1,\cdots,m): \bR^n\times (0,\infty)\rightarrow(0,\infty)$ be Lebesgue measurable functions satisfying condition (\ref{condition-2}).
	If $\vec{b}\in {\rm BMO}^m$, then $M^{\vec{b}}_{m}$ and $M^{\vec{b}}_{m,i}$ are both bounded from product spaces $M^{\varphi_1}_{\vec{q_1}}(\bR^n)\times\cdots\times M^{\varphi_m}_{\vec{q_m}}(\bR^n)$ to
	$M^{\varphi}_{\vec{q}}(\bR^n)$.
\end{theorem}
\begin{proof}
	Obviously, we only need to prove the boundedness of $M^{\vec{b}}_{m,i}$. From Theorem \ref{App-T-2}, $M^{\vec{b}}_{m,i}$ is bounded from $L^{\vec{q_1}}(\bR^n)\times\cdots\times L^{\vec{q_m}}(\bR^n)$ to $L^{\vec{q}}(\bR^n)$. Therefore, the conditions of Theorem \ref{T-T3} are fulfilled for  $M^{\vec{b}}_{m,i}$. As a consequence of Theorem \ref{T-T3}, $M^{\vec{b}}_{m,i}$ is bounded from $M^{\varphi_1}_{\vec{q_1}}(\bR^n)\times\cdots\times M^{\varphi_m}_{\vec{q_m}}(\bR^n)$ to
	$M^{\varphi}_{\vec{q}}(\bR^n)$. 
\end{proof}
Similarly, for multilinear Calder{\'o}n-Zygmund operators and their commutators, we have the following conclusions.
\begin{theorem}\label{C-3}
	Let $m\geq2$, $1<\vec{q_i}<\infty$ for all $i=1,\cdots,m$, $1/\vec{q}=1/\vec{q_1}+\cdots+1/\vec{q_m}$, and $\varphi,\varphi_i~(i=1,\cdots,m): \bR^n\times (0,\infty)\rightarrow(0,\infty)$ be Lebesgue measurable functions satisfying condition (\ref{condition-1}).
	Then $K_{m}$ is bounded from product spaces $M^{\varphi_1}_{\vec{q_1}}(\bR^n)\times\cdots\times M^{\varphi_m}_{\vec{q_m}}(\bR^n)$ to
	$M^{\varphi}_{\vec{q}}(\bR^n)$.
\end{theorem}
\begin{proof}
	From Theorem \ref{T-4-2}, we obtain the boundedness of $K_{m}$ from $L^{\vec{q_1}}(\bR^n)\times\cdots\times L^{\vec{q_m}}(\bR^n)$ to $L^{\vec{q}}(\bR^n)$. Therefore, the conditions of Theorem \ref{T-T1} are fulfilled for the operator $K_{m}$. As a consequence of Theorem \ref{T-T1}, $K_{m}$ is bounded from $M^{\varphi_1}_{\vec{q_1}}(\bR^n)\times\cdots\times M^{\varphi_m}_{\vec{q_m}}(\bR^n)$ to
	$M^{\varphi}_{\vec{q}}(\bR^n)$.
\end{proof}
\begin{theorem}\label{C-4}
	Let $m\geq2$, $1<\vec{q_k}<\infty$ for all $k=1,\cdots,m$, $1/\vec{q}=1/\vec{q_1}+\cdots+1/\vec{q_m}$, and $\varphi,\varphi_k~(k=1,\cdots,m): \bR^n\times (0,\infty)\rightarrow(0,\infty)$ be Lebesgue measurable functions satisfying condition (\ref{condition-2}).
	If $\vec{b}\in {\rm BMO}^m$, then $K^{\vec{b}}_{m}$ and $K^{\vec{b}}_{m,i}$ are both bounded from product spaces $M^{\varphi_1}_{\vec{q_1}}(\bR^n)\times\cdots\times M^{\varphi_m}_{\vec{q_m}}(\bR^n)$ to
	$M^{\varphi}_{\vec{q}}(\bR^n)$.
\end{theorem}
\begin{proof}
	From Theorem \ref{T-4-3}, we obtain the boundedness of $K^{\vec{b}}_{m,i}$ from $L^{\vec{q_1}}(\bR^n)\times\cdots\times L^{\vec{q_m}}(\bR^n)$ to $L^{\vec{q}}(\bR^n)$. Therefore, the conditions of Theorem \ref{T-T3} are fulfilled for the operator $K^{\vec{b}}_{m,i}$. As a consequence of Theorem \ref{T-T3}, $K^{\vec{b}}_{m,i}$ is bounded from $M^{\varphi_1}_{\vec{q_1}}(\bR^n)\times\cdots\times M^{\varphi_m}_{\vec{q_m}}(\bR^n)$ to
	$M^{\varphi}_{\vec{q}}(\bR^n)$. The boundedness of $K^{\vec{b}}_{m}$ is just an iteration of the boundedness  of $K^{\vec{b}}_{m,i}~(i=1,\cdots,m)$.
\end{proof}

\section*{Data availability statement}
Data sharing not applicable to this article as no datasets were generated or analysed during the current study.

\section*{Acknowledgements}
The author would like to express his deep gratitude to the anonymous referees for their careful reading of the manuscript and their comments and suggestions. This work is supported by the Natural Science Foundation of Henan Province of China (No. 202300410338) and the Nanhu Scholar Program for Young Scholars of Xinyang Normal University.\\

\end{document}